\documentclass[a4paper,12pt]{article}
\usepackage{amscd}
\usepackage{amsmath,amsfonts,amssymb,amscd}
\usepackage{indentfirst,graphicx,epsfig}
\usepackage{graphicx,psfrag}
\input{epsf}

\newtheorem{thm}{Theorem}

\newtheorem{claim}{Claim}

\newtheorem{cor}{Corollary}
\newenvironment {proof} {\noindent{\em Proof.}}{\hspace*{\fill}$\Box$\par\vspace{4mm}}

\def\qed{\hfill \nopagebreak\rule{5pt}{8pt}}

\title{\bf Hypergraph Tur\'an numbers of
vertex disjoint cycles\footnote{Supported by NSFC and the ``973" program. } }
\author{
\small  Ran Gu, Xueliang Li, Yongtang Shi\\
\small Center for Combinatorics and LPMC-TJKLC \\
\small Nankai University, Tianjin 300071, China \\
\small guran323@163.com, lxl@nankai.edu.cn,  shi@nankai.edu.cn
\date{}}
\begin{document}
\maketitle
\begin{abstract}
The Tur\'an number of a $k$-uniform hypergraph $H$, denoted by
$e{x_k}\left( {n;H} \right)$, is the maximum number of edges in any
$k$-uniform hypergraph $F$ on $n$ vertices which does not contain
$H$ as a subgraph. Let $\mathcal{C}_{\ell}^{\left( k \right)}$
denote the family of all $k$-uniform minimal cycles of length
$\ell$, $\mathcal{S}(\ell_1,\ldots,\ell_r)$ denote the family of
hypergraphs consisting of unions of $r$ vertex disjoint minimal
cycles of length $\ell_1,\ldots,\ell_r$, respectively, and
$\mathbb{C}_{\ell}^{\left( k \right)}$ denote a $k$-uniform linear
cycle of length $\ell$. We determine precisely $e{x_k}\left(
{n;\mathcal{S}(\ell_1,\ldots,\ell_r)} \right)$ and $e{x_k}\left(
{n;\mathbb{C}_{{\ell_1}}^{\left( k \right)}, \ldots,
\mathbb{C}_{{\ell_r}}^{\left( k \right)}} \right)$ for sufficiently
large $n$. The results extend recent results of F\"{u}redi and Jiang
who determined the Tur\'an numbers for single $k$-uniform minimal
cycles and linear cycles.
\\[2mm]
[2mm] \textbf{AMS Subject Classification (2010):} 05D05, 05C35, 05C65
\end{abstract}

\section{Introduction}
In this paper, we employ standard definitions and notation
from hypergraph theory  (see e.g.,\cite{Ber}).
A \emph{hypergraph} is a pair $H = (V,E)$ consisting of a set $V$
 of vertices and a set $E \subseteq \mathcal{P}(V)$ of edges.
If every edge contains exactly $k$ vertices, then $H$ is a {\it
$k$-uniform hypergraph}. For two hypergraphs $G$ and $H$, we write
$G\subseteq H$ if there is an injective homomorphism from $G$ into
$H$. We use $G \cup  H$ to denote the disjoint union of
(hyper)graphs $G$ and $H$. By disjoint, we will always mean vertex
disjoint. A {\it Berge path} of length $\ell$ is a family of
distinct sets $\{F_1, \ldots, F_{\ell}\}$ and $\ell + 1$ distinct
vertices $v_1, \ldots, v_{\ell+1}$ such that for each $1\leq
i\leq\ell$, $F_i$ contains $v_i$ and $v_{i+1}$. Let
$\mathcal{B}_\ell^{(k)}$ denote the family of $k$-uniform Berge
paths of length $\ell$. A {\it linear path} of length $\ell$ is a
family of sets $\{F_1, \ldots, F_{\ell}\}$ such that $|F_i\cap
F_{i+1}| = 1$ for each $i$ and $F_i\cap F_j =\emptyset$ whenever $|i
-j|> 1$. Let $\mathbb{P}_{\ell}^{(k)}$ denote the $k$-uniform linear
path of length $\ell$. It is unique up to isomorphisms. A
$k$-uniform \emph{Berge cycle} of length $\ell$ is a cyclic list of
distinct $k$-sets $A_1,\ldots , A_\ell$ and $\ell$ distinct vertices
$v_1, \ldots, v_\ell$ such that for each $1 \leq i \leq \ell$, $A_i$
contains $v_i$ and $v_{i+1}$ (where $v_{\ell+1} = v_1$). A
$k$-uniform \emph{minimal cycle} of length $\ell$ is a cyclic list
of $k$-sets $A_1,\ldots , A_\ell$ such that consecutive sets
intersect in at least one element and nonconsecutive sets are
disjoint. Denote the family of all $k$-uniform minimal cycles of
length $\ell$ by $\mathcal{C}_{\ell}^{\left( k \right)}$. A
$k$-uniform \emph{linear cycle} of length $\ell$, denoted by
$\mathbb{C}_{\ell}^{\left( k \right)}$, is a cyclic list of $k$-sets
$A_1,\ldots , A_\ell$ such that consecutive sets intersect in
exactly one element and nonconsecutive sets are disjoint.

The {\it Tur\'an number}, or extremal number, of a $k$-uniform
hypergraph $H$, denoted by $ex_k(n;H)$, is the maximum number of
edges in any $k$-uniform hypergraph $F$ on $n$ vertices which does
not contain $H$ as a subgraph. This is a natural generalization of
the classical Tur\'an number for $2$-uniform graphs; we restrict
ourselves to the case of $k$-uniform hypergraphs. Let
$ex_k(n;F_1,F_2,\ldots,F_r)$ denote the $k$-uniform hypergraph
Tur\'an Number of a list of $k$-uniform hypergraphs
$F_1,F_2,\ldots,F_r$, i.e.,
$ex_k(n;F_1,F_2,\ldots,F_r)=ex_k(n;F_1\cup F_2\cup \ldots\cup F_r)$.

For the class of $k$-uniform Berge paths of length $\ell$, Gy\"{o}ri
et al \cite{GKL} determined $ex_k(n;\mathcal{B}_\ell^{(k)})$ exactly
for infinitely many $n$. In \cite{FJS}, F\"{u}redi et al.
established the following results.

\begin{thm}\label{FJS}\cite{FJS}
Let $k$, $t$ be positive integers, where $k\geq 3$. For sufficiently
large $n$, we have
\[e{x_k}\left( n;\mathbb{P}_{2t+1}^{\left( k \right)} \right) =\binom{n-1}{k-1}+
\binom{n-2}{k-1}+\ldots+\binom{n-t}{k-1}.\] The only extremal family
consists of all the $k$-sets in $[n]$ that meet some fixed set $S$
of $t$ vertices. Also,
\[e{x_k}\left( n;\mathbb{P}_{2t+2}^{\left( k \right)} \right) =\binom{n-1}{k-1}+
\binom{n-2}{k-1}+\ldots+\binom{n-t}{k-1}+\binom{n-t-2}{k-2}.\] The
only extremal family consists of all the $k$-sets in $[n]$ that meet
some fixed set $S$ of $t$ vertices plus all the $k$-sets in
$[n]\setminus S$ that contain some two fixed elements.
\end{thm}
For more results we refer to \cite{FJS, MV}.

For the minimal and linear cycles, F\"{u}redi and Jiang \cite{FJ},
determined the extremal numbers when the forbidden hypergraph is a
single minimal cycle or a single linear cycle. This confirms, in a
stronger form, a conjecture of Mubayi and Verstra\"ete \cite{MV} for
$k\geq 5$ and adds to the limited list of hypergraphs whose Tur\'an
numbers have been known either exactly or asymptotically. Their main
results are as follows:

\begin{thm}\label{FJ1}\cite{FJ}
Let $t$ be a positive integer, $k\geq4$. For sufficiently large $n$, we have\\
$e{x_k}\left( n;\mathcal{C}_{2t+1}^{\left( k \right)} \right) = \left( {\begin{array}{*{20}{c}}
n\\
k
\end{array}} \right) - \left( {\begin{array}{*{20}{c}}
{n - t}\\
k
\end{array}} \right)$, and $e{x_k}\left( n;\mathcal{C}_{2t+2}^{\left( k \right)}
\right) = \left( {\begin{array}{*{20}{c}}
n\\
k
\end{array}} \right) - \left( {\begin{array}{*{20}{c}}
{n - t}\\
k
\end{array}} \right)+1$.
For $\mathcal{C}_{2t+1}^{\left( k \right)}$, the only extremal
family consists of all the $k$-sets in $[n]$ that meet some fixed
$k$-set $S$. For $\mathcal{C}_{2t+2}^{\left( k \right)}$, the only
extremal family consists of all the $k$-sets in $[n]$ that intersect
some fixed $t$-set $S$ plus one additional $k$-set outside $S$.
\end{thm}

\begin{thm}\label{FJ2}\cite{FJ}
Let $t$ be a positive integer, $k\geq5$. For sufficiently large $n$, we have\\
$e{x_k}\left( n;\mathbb{C}_{2t+1}^{\left( k \right)} \right) = \left( {\begin{array}{*{20}{c}}
n\\
k
\end{array}} \right) - \left( {\begin{array}{*{20}{c}}
{n - t}\\
k
\end{array}} \right)$, and $e{x_k}\left( n,\mathbb{C}_{2t+2}^{\left( k \right)} \right) =
\left( {\begin{array}{*{20}{c}}
n\\
k
\end{array}} \right) - \left( {\begin{array}{*{20}{c}}
{n - t}\\
k
\end{array}} \right)+\left( {\begin{array}{*{20}{c}}
{n - t - 2}\\
{k - 2}
\end{array}} \right)$.
For $\mathbb{C}_{2t+1}^{\left( k \right)}$, the only extremal family
consists of all the $k$-sets in $[n]$ that meet some fixed $k$-set
$S$. For $\mathbb{C}_{2t+2}^{\left( k \right)}$, the only extremal
family consists of all the $k$-sets in $[n]$ that intersect some
fixed $t$-set $S$ plus all the $k$-sets in $[n]\setminus S$ that
contain some two fixed elements.
\end{thm}

From the definition of $k$-uniform minimal cycles, two $k$-uniform minimal cycles of
the same length may not be isomorphic. Hence we define the following family of hypergraphs,
where every member consists of $r$ vertex disjoint cycles:
\[\mathcal{S}(\ell_1,\ldots,\ell_r)=\{C_1 \cup \ldots\cup C_r:C_i\in \mathcal{C}_{\ell_i}^{(k)}
~for~i\in [r]\}  \]

Apart from the results above, we will need the following two results:

\begin{thm}\label{KMW}\cite{KMW}
Let $H$ be a $k$-uniform hypergraph
on $n$ vertices with no two edges intersecting in
exactly one vertex, where $k \geq 3$. Then
$|E(H)|\leq \binom{n}{k-2}$.
\end{thm}

We build on earlier work of  F\"{u}redi and Jiang \cite{FJ}, in this paper, we
determine precisely the exact Tur\'an numbers when forbidden hypergraphs are
$r$  vertex disjoint minimal cycles or $r$ vertex disjoint linear cycles.
Our main results are as follows:
\begin{thm}\label{th1}
Let integers $k\geq4$, $r\geq1$, $\ell_1,\ldots,\ell_r\geq3$,  $t =
\sum\limits_{i = 1}^r {\left\lfloor {\frac{{{\ell _i} + 1}}{2}}
\right\rfloor }  - 1$, and $I=1$, if all the $\ell_1,\ldots,\ell_r$
are even, $I=0$ otherwise. For sufficiently large $n$,
\[e{x_k}\left( {n;\mathcal{S}(\ell_1,\ldots,\ell_r)} \right) = \left( {\begin{array}{*{20}{c}}
n\\
k
\end{array}} \right) - \left( {\begin{array}{*{20}{c}}
{n - t}\\
k
\end{array}} \right) + I.\]
\end{thm}

\begin{thm}\label{th2}
Let integers $k\geq5$, $r\geq1$, $\ell_1,\ldots,\ell_r\geq3$,  $t=\sum\limits_{i = 1}^r  {\left\lfloor
{\frac{{{\ell_i} + 1}}{2}} \right\rfloor }  - 1$, and $J=\left( {\begin{array}{*{20}{c}}
{n - t - 2}\\
{k - 2}
\end{array}} \right)$, if all the $\ell_1,\ldots,\ell_r$ are even, $J=0$ otherwise.
For sufficiently large $n$,
\[e{x_k}\left( {n;\mathbb{C}_{{\ell_1}}^{\left( k \right)},\ldots, \mathbb{C}_{{\ell_r}}^{\left(
k \right)}} \right) = \left( {\begin{array}{*{20}{c}}
n\\
k
\end{array}} \right) - \left( {\begin{array}{*{20}{c}}
{n - t}\\
k
\end{array}} \right) + J.\]
\end{thm}

Sometimes, we allow the hypergraph to contain less than $r$ minimal
or linear cycles, consider the Tur\'an number in such cases, we have
the following two corollaries. We use notation $r\cdot F$ to denote
$r$ vertex disjoint copies of hypergraph $F$. Let
$\ell_1=\ldots=\ell_r=\ell$, we can immediately get the following
two corollaries from Theorems \ref{th1} and \ref{th2}.
\begin{cor}\label{cor1}
Let integers $k\geq4$, $r\geq1$, $\ell\geq3$,  $t =r\left\lfloor
{\frac{{{\ell} + 1}}{2}} \right\rfloor - 1$, and $I=1$, if $\ell$ is
even, $I=0$, if $\ell$ is odd. For sufficiently large $n$,
\[e{x_k}\left( {n;r\cdot \mathcal{C}_{\ell}^{(k)}} \right) = \left( {\begin{array}{*{20}{c}}
n\\
k
\end{array}} \right) - \left( {\begin{array}{*{20}{c}}
{n - t}\\
k
\end{array}} \right) + I.\]
\end{cor}

\begin{cor}\label{cor2}
Let integers $k\geq5$, $r\geq1$, $\ell\geq3$,  $t =r\left\lfloor
{\frac{{{\ell} + 1}}{2}} \right\rfloor - 1$, and $J=\left(
{\begin{array}{*{20}{c}}
{n - t - 2}\\
{k - 2}
\end{array}} \right)$, if $\ell$ is even, $J=0$, if $\ell$ is odd.
For sufficiently large $n$,
\[e{x_k}\left( {n;r\cdot \mathbb{C}_{{\ell}}^{\left( k \right)} }\right) = \left( {\begin{array}{*{20}{c}}
n\\
k
\end{array}} \right) - \left( {\begin{array}{*{20}{c}}
{n - t}\\
k
\end{array}} \right) + J.\]
\end{cor}

\section{Proof of Theorem \ref{th1}}
For convenience, we define $f(n,k,\{\ell_1,\ldots,\ell_r\})=\left(
{\begin{array}{*{20}{c}}
n\\
k
\end{array}} \right) - \left( {\begin{array}{*{20}{c}}
{n - t}\\
k
\end{array}} \right) + I$. Note that the hypergraph on $n$ vertices that has every edge
incident to some fixed $t$-set $S$, along with one additional edge disjoint from $S$
when all of $\ell_1,\ldots,\ell_r$ are  even, has exactly  $f(n,k,\{\ell_1,\ldots,\ell_r\})$
edges and dose not contain a copy of any member of $\mathcal{S}(\ell_1,\ldots,\ell_r)$.

Thus, to prove Theorem \ref{th1}, it suffices to prove that
$e{x_k}\left( {n;\mathcal{S}(\ell_1,\ldots,\ell_r)} \right) \leq \left( {\begin{array}{*{20}{c}}
n\\
k
\end{array}} \right) - \left( {\begin{array}{*{20}{c}}
{n - t}\\
k
\end{array}} \right) + I$, i.e.,
any hypergraph on $n$ vertices with more than $f(n,k,\{\ell_1,\ldots,\ell_r\})$
edges must contain a member of $\mathcal{S}(\ell_1,\ldots,\ell_r)$.
We use induction on $r$.  From Theorem \ref{FJ1}, the case $r=1$ has been proved.
Assume that $r\geq 2$, and Theorem \ref{th1} holds for smaller $r$.

Let $H$ be a hypergraph on $n$ vertices with $m$ edges and $m>f(n,k,\{\ell_1,\ldots,\ell_r\})$.

Since $f(n,k,\{\ell_1,\ldots,\ell_r\})>f(n,k,\ell_1)$ for sufficiently large $n$, there exists
at least one $k$-uniform minimal $\ell_1$-cycle in $H$. Take one of them, denote its vertex set by
$C$, so $\ell_1\leq |C|\leq(k-1)\ell_1$.
We have that $|E(H\setminus C)|\leq f(n-|C|,k,\{\ell_2,\ldots,\ell_r\})$, since otherwise,
by induction hypothesis, we can find vertex disjoint copies of $\mathcal{C}_{\ell_2}^{(k)}\cup \ldots\cup
\mathcal{C}_{\ell_r}^{(k)}$ in $H$; plus the minimal $\ell_1$-cycle on $C$, there is a copy of a
member of $\mathcal{S}(\ell_1,\ldots,\ell_r)$ in $H$ already.

Let $m_C$ denote the number of edges in $H$ incident to vertices in $C$. Then,
\begin{align}
m_C &\geq  m-f(n-|C|,k,\{\ell_2,\ldots,\ell_r\}) \\
\null  &\geq f(n,k,\{\ell_1,\ldots,\ell_r\})-f(n-\ell_1,k,\{\ell_2,\ldots,\ell_r\})\\
\null  &= \frac{{\left\lfloor {\frac{{{\ell_1} + 1}}{2}} \right\rfloor }}{{\left( {k - 1}
 \right)!}}{n^{k - 1}} + O\left( {{n^{k - 2}}} \right)\label{mC}.
\end{align}
We call an edge in $H$ is a \emph{terminal edge} if it contains exactly one vertex in $C$.
Let $T$ denote the set of all terminal edges in $H$. For every $(k-1)$-set $R$ in $V(H)\setminus C$,
define \[T_R=\{E \in T:~R\subseteq E\}.\]
According to the size of each set $T_R$, we partite all the $(k-1)$-sets in $V(H)\setminus C$
into two sets, such that:
\[X=\{R \subseteq V(H)\setminus C~and~|R|=k-1:|T_R|\leq \left\lfloor {\frac{{{\ell_1} + 1}}{2}}
\right\rfloor  - 1\}\]
\[Y=\{R \subseteq V(H)\setminus C~and~|R|=k-1:|T_R|\geq \left\lfloor {\frac{{{\ell_1} + 1}}{2}}
\right\rfloor \}.\]
It is not difficult to give an upper bound of $m_C$ with the terms $|X|$ and $|Y|$ as follows:
\begin{eqnarray*}
m_C &\leq & \left( {\begin{array}{*{20}{c}}
{|C|}\\
2
\end{array}} \right)\left( {\begin{array}{*{20}{c}}
{n-2}\\
{k - 2}
\end{array}} \right) + |X|\left( {\left\lfloor {\frac{{{\ell_1} + 1}}{2}}
\right\rfloor  - 1} \right) + |Y| \cdot |C|\\\
\null  &\leq& \left( {\begin{array}{*{20}{c}}
{|C|}\\
2
\end{array}} \right)\left( {\begin{array}{*{20}{c}}
{n-2}\\
{k - 2}
\end{array}} \right) + \left( {\begin{array}{*{20}{c}}
n\\
{k - 1}
\end{array}} \right)\left( {\left\lfloor {\frac{{{\ell_1} + 1}}{2}}
\right\rfloor  - 1} \right) + |Y| \cdot {\ell_1}\left( {k - 1} \right).
\end{eqnarray*}
Combine with (\ref{mC}), we have that
\begin{equation}\label{eqY}
|Y| \ge \frac{{{n^{k - 1}}}}{{\left( {k - 1} \right){\ell_1}\left( {k - 1} \right)!}}+O(n^{k-2}).
\end{equation}
For any $(k-1)$-set $R\in Y$, there are at least $\left\lfloor {\frac{{{\ell_1} + 1}}{2}} \right\rfloor$
vertices in $C$ that can form terminal edges with $R$. We choose exactly
$\left\lfloor {\frac{{{\ell_1} + 1}}{2}} \right\rfloor$ of them, call the vertex set of these
$\left\lfloor {\frac{{{\ell_1} + 1}}{2}} \right\rfloor$  vertices \emph{terminal set} relative to $R$.
Since the number of  $\left\lfloor {\frac{{{\ell_1} + 1}}{2}} \right\rfloor$-sets in $C$
is at most $\left( {\begin{array}{*{20}{c}}
{|C|}\\
{\left\lfloor {\frac{{{\ell_1} + 1}}{2}} \right\rfloor }
\end{array}} \right)$, we can get that
some elements in $Y$ may have the same terminal set. And it is easy
to derive that, the number of $(k-1)$-sets in $Y$ with the same
terminal set, is at least
\[\frac{{{n^{k - 1}}}}{{\left( {k - 1} \right){\ell _1}\left( {k - 1}
\right)!}}{\binom{|C|}{\left\lfloor {\frac{{{\ell _1} + 1}}{2}}
\right\rfloor}^{{\rm{ - }}1}} + O({n^{k - 2}}) \ge \frac{{{n^{k -
1}}}}{{\left( {k - 1} \right){\ell _1}\left( {k - 1} \right)!}}{
\binom{{\left( {k - 1} \right){\ell _1}}}{{\left\lfloor
{\frac{{{\ell _1} + 1}}{2}} \right\rfloor }}^{{\rm{ - }}1}} +
O({n^{k - 2}}).\]

Choose one terminal set $U$ in $C$, such that there are at least
$\frac{{{n^{k - 1}}}}{{\left( {k - 1} \right){\ell _1}\left( {k - 1} \right)!}}
{ \binom{{\left( {k - 1} \right){\ell _1}}}{{\left\lfloor {\frac{{{\ell _1} + 1}}{2}}
\right\rfloor }} ^{{\rm{ - }}1}} + O({n^{k - 2}})$ $(k-1)$-sets in $V(H)\setminus C$,
every such $(k-1)$-set can form a terminal edge with every vertex in $U$.
Let $R_U$ be the set of all the common $(k-1)$-sets associate with $U$ in $V(H)\setminus C$,
we have that
\begin{equation}\label{RU}
|R_U|\geq \frac{{{n^{k - 1}}}}{{\left( {k - 1} \right){\ell _1}\left( {k - 1} \right)!}}
\binom{{\left( {k - 1} \right){\ell _1}}}{{\left\lfloor {\frac{{{\ell _1} + 1}}{2}}
\right\rfloor }} ^{{\rm{ - }}1} + O({n^{k - 2}}).
\end{equation}

Let $m_U$ denote the number of edges incident to vertices in $U$,
then,
\[m_U\leq \left\lfloor {\frac{{{\ell _1} + 1}}{2}} \right\rfloor
\binom{n-\left\lfloor {\frac{{{\ell _1} + 1}}{2}} \right\rfloor}{k-1}+m',\]
where $m'$ is the number of edges which contain at least two vertices in $U$.
With some calculations, we have that
\begin{eqnarray*}
&\null & f(n,k,\{\ell_1,\ldots,\ell_r\})-f(n-\left\lfloor {\frac{{{\ell _1} + 1}}{2}}
\right\rfloor,k,\{\ell_2,\ldots,\ell_r\})-m_U\\
&=& \binom{n-1}{k-1}+\binom{n-2}{k-1}+\cdots+\binom{n-\left\lfloor {\frac{{{\ell _1}
+ 1}}{2}} \right\rfloor}{k-1}-m_U\\
&\geq& \left[\binom{n-1}{k-1}-\binom{n-\left\lfloor {\frac{{{\ell _1}
+ 1}}{2}} \right\rfloor}{k-1}\right]+\left[\binom{n-2}{k-1}-\binom{n-\left\lfloor {\frac{{{\ell _1}
+ 1}}{2}} \right\rfloor}{k-1}\right]\\
&\null&+\cdots+\left[\binom{n-\left\lfloor {\frac{{{\ell _1}
+ 1}}{2}} \right\rfloor+1}{k-1}-\binom{n-\left\lfloor {\frac{{{\ell _1}
+ 1}}{2}} \right\rfloor}{k-1}\right]-m'.
\end{eqnarray*}
It is not difficult to deduce that the last  expression
is no less than zero (consider the combinatorial meaning of that
expression), hence, we can derive that
\begin{eqnarray*}
E(H\setminus U) &= & m-m_U>f(n,k,\{\ell_1,\ldots,\ell_r\})-m_U\\
\null  &\geq&  f(n-\left\lfloor {\frac{{{\ell_1} + 1}}{2}} \right\rfloor ,k,\{\ell_2,\ldots,\ell_r\}).
\end{eqnarray*}
Thus by the induction hypothesis, there exists a member of
$\mathcal{S}(\ell_2,\ldots,\ell_r)$ with vertex set $W$ in
$V(H)\setminus U$, also we have that
\begin{equation}\label{W}
|W|\leq \left( {k - 1} \right)\sum\limits_{i = 2}^r {{\ell_i}}.
\end{equation}

Now we focus on finding a $k$-uniform minimal $\ell_1$-cycle disjoint from $W$.

Considering the $(k-1)$-uniform hypergraph $H_0$ with vertex set $V(H)\setminus U$ and edge set $R_U$,
we will prove the following claim:
\begin{claim}\label{claim1}
There are ${\left\lfloor {\frac{{{\ell _1}}}{2}} \right\rfloor}$ pairs of $(k-1)$-edges in $H_0$,
say $\{a_i,b_i\}$, $i=1,\ldots,{\left\lfloor {\frac{{{\ell _1}}}{2}} \right\rfloor}$,
such that for every $i$, $a_i$ and $b_i$ have exactly one common vertex, and for any $j\neq i$,
$\{a_i,b_i\}$ and $\{a_j,b_j\}$ are vertex disjoint, moreover, all these $(k-1)$-edges disjoint from $W$.
\end{claim}

\begin{proof}
The number of $(k-1)$-edges incident with vertices in $W$ is at most $|W|\cdot \binom{n - 1}{k - 2}$.
With the aid of (\ref{RU}) and (\ref{W}), in $R_U$, the number of $(k-1)$-edges disjoint from $W$ is at least
\[\frac{{{n^{k - 1}}}}{{\left( {k - 1} \right){\ell _1}\left( {k - 1} \right)!}}
\binom{{\left( {k - 1} \right){\ell _1}}}{{\left\lfloor
{\frac{{{\ell _1} + 1}}{2}} \right\rfloor }}^{{\rm{ - }}1} + O({n^{k
- 2}}) -(k-1)\sum\limits_{i = 2}^r {{\ell _i}}\binom{n-1}{k-2}
> \binom{n-\lfloor {\frac{{{\ell_1} + 1}}{2}}\rfloor}{k-2}.\]
By Theorem \ref{KMW}, we can find a pair of $(k-1)$-edges
$\{a_1,b_1\}$ with exactly one common vertex. Let $p=\left\lfloor
{\frac{{{\ell _1}}}{2}} \right\rfloor \left( {2k - 3} \right)$,
since $\frac{{{n^{k - 1}}}}{{\left( {k - 1} \right){\ell _1}\left(
{k - 1} \right)!}}\binom{{\left( {k - 1} \right){\ell
_1}}}{{\left\lfloor {\frac{{{\ell _1} + 1}}{2}} \right\rfloor }}
^{{\rm{ - }}1} + O({n^{k - 2}}) -(k-1)\sum\limits_{i = 2}^r {{\ell
_i}}\binom{n-1}{k-2}-p\binom{n-1}{k-2}
> \binom{n-\lfloor {\frac{{{\ell_1} + 1}}{2}}\rfloor}{k-2}$,
we can repeat the argument above to find $\{a_2,b_2\}$,
$\ldots,\{a_{\left\lfloor {\frac{{{\ell _1}}}{2}}
\right\rfloor},b_{\left\lfloor {\frac{{{\ell _1}}}{2}}
\right\rfloor}\}$ satisfying the properties described in Claim
\ref{claim1}.
\end{proof}

Let $U=\{u_1,\ldots,u_{\left\lfloor {\frac{{{\ell _1} + 1}}{2}} \right\rfloor}\}$.
To form the required minimal $\ell_1$-cycle, we need to consider such two cases:

{\bf Case 1.} $\ell_1$ is even.

Find ${\frac{{{\ell _1}}}{2}}$ pairs of $(k-1)$-edges in $H_0$ as described in Claim \ref{claim1},
still denote them by $\{a_i,b_i\}$, $i=1,\ldots,\frac{{{\ell _1}}}{2}$.
Construct a $k$-uniform minimal $\ell_1$-cycle in $H$ with edges:
\[a_1\cup \{u_1\}, b_1\cup \{u_2\}, a_2\cup \{u_2\},\ldots, b_{{\frac{{{\ell _1}}}{2}}-1}
\cup \{u_{{\frac{{{\ell _1}}}{2}}}\},
a_{{\frac{{{\ell _1}}}{2}}}\cup \{u_{{\frac{{{\ell _1}}}{2}}}\},b_{{\frac{{{\ell _1}}}{2}}}\cup \{u_1\}.\]

{\bf Case 2.} $\ell_1$ is odd.

Find ${\frac{{{\ell _1}-3}}{2}}$ pairs of $(k-1)$-edges in $H_0$ as
described in Claim \ref{claim1}. Similar to the proof of Claim
\ref{claim1}. Let $Q$ be the vertex set of $W$ and all these
${\frac{{{\ell _1}-3}}{2}}$ pairs of $(k-1)$-edges, hence, $|Q|=
{\frac{{{\ell _1-3}}}{2}}(2k-3)+|W|$. By Theorem \ref{FJS},
$e{x_{k-1}}\left( n-\left\lfloor {\frac{{{\ell _1} + 1}}{2}}
\right\rfloor; \mathbb{P}_{3}^{\left(
k-1\right)}\right)=\frac{1}{{\left( {k - 2} \right)!}}{n^{k - 2}} +
O({n^{k - 3}})$, for sufficiently large $n$. In $H_0$, the number of
$(k-1)$-edges disjoint from $Q$ is at least $\frac{{{n^{k -
1}}}}{{\left( {k - 1} \right){\ell _1}\left( {k - 1}
\right)!}}\binom{{\left( {k - 1} \right){\ell _1}}}{{\left\lfloor
{\frac{{{\ell _1} + 1}}{2}} \right\rfloor }} ^{{\rm{ - }}1} +
O({n^{k - 2}}) -|Q|\binom{n-1}{k-2}>\frac{1}{{\left( {k - 2}
\right)!}}{n^{k - 2}} + O({n^{k - 3}})$. That implies in $H_0$, we
can find a $\mathbb{P}_{3}^{\left( k-1 \right)}$ in remaining
$(k-1)$-edges disjoint from $Q$. Let $x,y,z$ be the three
consecutive $(k-1)$-edges in that $\mathbb{P}_{3}^{\left( k-1
\right)}$, then, in $H$, we can form a $k$-uniform minimal
$\ell_1$-cycle with edges:
\[a_1\cup \{u_1\}, b_1\cup \{u_2\}, a_2\cup \{u_2\},\ldots, a_{{\frac{{{\ell _1-3}}}{2}}}
\cup \{u_{{\frac{{{\ell _1-3}}}{2}}}\},\]
\[b_{{\frac{{{\ell _1-3}}}{2}}}\cup \{u_{{\frac{{{\ell _1-1}}}{2}}}\},
x\cup \{u_{{\frac{{{\ell _1-1}}}{2}}}\}, y\cup \{u_{{\frac{{{\ell _1+1}}}{2}}}\},z\cup \{u_1\}.\]
Moreover, it is easy to see that this $k$-uniform minimal $\ell_1$-cycle is not only
minimal, but also linear, whenever $\ell_1$ is even or odd.
Thus, we have constructed $r$ disjoint $k$-uniform minimal cycles.
So the hypergraph  which contains no member of $\mathcal{S}(\ell_1,\ldots,\ell_r)$
can not have more than $f(n,k,\{\ell_1,\ldots,\ell_r\})$ edges.
Thus completes the proof.\qed

\section{Proof of Theorem \ref{th2}}
The argument in the proof of Theorem \ref{th2} is similar to
the proof of Theorem \ref{th1}.

Let $g(n,k,\{\ell_1,\ldots,\ell_r\})=\left( {\begin{array}{*{20}{c}}
n\\
k
\end{array}} \right) - \left( {\begin{array}{*{20}{c}}
{n - t}\\
k
\end{array}} \right) + J.$
Firstly, we point out that the hypergraph on $n$  vertices that has
every edge incident to some fixed $t$-set $S$, along with all the
$k$-edges disjoint from $S$ containing some two fixed elements not
in $S$ when all of $\ell_1,\ldots,\ell_r$ are  even, has exactly
$g(n,k,\{\ell_1,\ldots,\ell_r\})$ edges and dose not contain a copy
of any member of $\mathbb{C}_{{\ell_1}}^{\left( k
\right)}\cup\ldots\cup \mathbb{C}_{{\ell_r}}^{\left( k \right)}$.

Hence it suffices to prove that $e{x_k}\left(
{n;\mathbb{C}_{{\ell_1}}^{\left( k \right)},\ldots,
\mathbb{C}_{{\ell_r}}^{\left( k \right)}} \right)\leq
g(n,k,\{\ell_1,\ldots,\ell_r\})$. We proceed by induction on $r$
again since the case $r=1$ is provided by Theorem \ref{FJ2}. Let $H$
be a hypergraph on $n$ vertices with
$m>g(n,k,\{\ell_1,\ldots,\ell_r\})$  edges. If one of
$\ell_1,\ldots,\ell_r$ is even, rearrange the sequence to make sure
$\ell_1$ is even.

As in the proof of Theorem \ref{th1}, since
$g(n,k,\{\ell_1,\ldots,\ell_r\})>g(n,k,\ell_1)$ for sufficiently
large $n$, there exists at least one $k$-uniform linear
$\ell_1$-cycle in $H$. Take one of them, denote its vertex set by
$C$. Similarly, we have that $|E(H\setminus C)|\leq
g(n-|C|,k,\{\ell_2,\ldots,\ell_r\})$. Still let $m_C$ denote the
number of edges in $H$ incident to vertices in $C$, with some
calculations, we can get that:
\[m_C \geq  \frac{{\left\lfloor {\frac{{{\ell_1} + 1}}{2}}
 \right\rfloor }}{{\left( {k - 1}
 \right)!}}{n^{k - 1}} + O\left( {{n^{k - 2}}} \right).\]
Again we define terminal edges, $T_R$, $X$, $Y$ as before, we can
find the $\left\lfloor {\frac{{{\ell_1} + 1}}{2}} \right\rfloor$-set
$U$, too. Then by induction hypothesis, we can find a copy of
$\mathbb{C}_{{\ell_2}}^{\left( k \right)}\cup\ldots\cup
\mathbb{C}_{{\ell_r}}^{\left( k \right)}$ on vertex set $W$ in
$V(H)\setminus U$.  With the same method used in the proof of
Theorem \ref{th1}, we can select a terminal set of size
$\left\lfloor {\frac{{{\ell_1} + 1}}{2}} \right\rfloor$ in $C$,
then, similarly, we can construct a $k$-uniform linear
$\ell_1$-cycle in $H$ since the $k$-uniform minimal $\ell_1$-cycle
we described in the proof of Theorem \ref{th1} is also linear. And
this $k$-uniform linear $\ell_1$-cycle avoid the vertices in $W$,
hence we know that the hypergraph  which contains no
$\mathbb{C}_{{\ell_1}}^{\left( k \right)}\cup\ldots\cup
\mathbb{C}_{{\ell_r}}^{\left( k \right)}$ can not have more than
$g(n,k,\{\ell_1,\ldots,\ell_r\})$ edges. The proof is thus complete.
\qed

\end{document}